\documentclass[a4paper,oneside,10pt
]{article}
\usepackage[utf8]{inputenc}
\usepackage[T1]{fontenc}
\usepackage{textcomp}
\usepackage{hyperref}
\usepackage{url}
\usepackage{amsmath,amssymb,amsthm}
\usepackage{fullpage}
\usepackage{mathrsfs}

\usepackage[all,cmtip]{xy}

\renewcommand\epsilon\varepsilon
\renewcommand\phi\varphi
\newcommand\RR{\mathbb{R}}
\newcommand\CC{\mathbb{C}}
\newcommand\ZZ{\mathbb{Z}}
\newcommand\QQ{\mathbb{Q}}

\newcommand\kArt{k\text{-}\mathbf{Art}}

\newcommand\Repr{\mathfrak{R}(\Gamma, G)}
\renewcommand\gg{\mathfrak{g}}
\newcommand\aab{\mathcal{A}^\bullet}
\newcommand\PropP{(P)}

\theoremstyle{definition}
\newtheorem{Def}{Definition}[section]
\theoremstyle{plain}

\newtheorem{Lem}[Def]{Lemma}
\newtheorem{The}[Def]{Theorem}
\newtheorem{Cor}[Def]{Corollary}
\theoremstyle{remark}

\newtheorem{Rem}[Def]{Remark}

\DeclareMathOperator\id{id}

\DeclareMathOperator\Ker{Ker}

\DeclareMathOperator\Hom{Hom}

\DeclareMathOperator\Ad{Ad}
\DeclareMathOperator\ad{ad}
\DeclareMathOperator\Gr{Gr}
\DeclareMathOperator\Obj{Obj}
\DeclareMathOperator\Iso{Iso}
\DeclareMathOperator\Dec{Dec}
\DeclareMathOperator\Alb{Alb}

\hypersetup{
pdfauthor = {Louis-Clément LEFÈVRE},
pdfsubject = {Article},
pdftitle = {A criterion for quadraticity of a representation of the fundamental group
of an algebraic variety},
pdfkeywords = {Representation scheme, Mixed Hodge theory, fundamental groups of algebraic varieties},
}

\title{A criterion for quadraticity of a representation of the fundamental group
of an algebraic variety}
\author{Louis-Clément Lefèvre}
\date{\today}

\begin{document}

\maketitle

\begin{abstract}
Let $\Gamma$ be a finitely presented group and $G$ a linear algebraic group over $\mathbb{R}$.
A representation $\rho:\Gamma\rightarrow G(\mathbb{R})$ can be seen as an $\mathbb{R}$-point of the representation variety
$\mathfrak{R}(\Gamma, G)$.
It is known from the work of Goldman and Millson that if $\Gamma$ is the fundamental group of a compact Kähler manifold
and $\rho$ has image contained in a compact subgroup
then $\rho$ is analytically defined by homogeneous quadratic equations in $\mathfrak{R}(\Gamma, G)$.
When $X$ is a smooth complex algebraic variety, we study a certain criterion under which
this same conclusion holds.
\end{abstract}

\section{Introduction}
Many restrictions are known on the question of wether a given finitely presented group $\Gamma$
can be obtained as the fundamental group of a compact Kähler manifold, and some restrictions are known
for smooth complex algebraic varieties.
See \cite{ABCKT} for an introduction to these questions.
One way to study these groups is via their representations into a linear algebraic group $G$
over $\RR$: there exists a scheme $\Repr$ parametrizing such representations $\rho$
(see section~\ref{SecRepresentationVariety})
and it is sometimes possible to describe
$\rho$ as a singularity in $\Repr$, up to analytic isomorphism (see section~\ref{SecAnalyticGerms}).

The first known theorem in this direction is obtained by Goldman and Millson in
\cite{GoldmanMillson}.

\begin{The}[Goldman-Millson]
Let $X$ be a compact Kähler manifold and $\Gamma$ its fundamental group.
Let $\rho:\Gamma\rightarrow G(\RR)$ be a representation with image contained in a compact subgroup of $G(\RR)$.
Then $(\Hom(\Gamma, G), \rho)$ is a quadratic singularity.
\end{The}

We will need to review this theory in section
\ref{SecGoldmanMillson}.
To pass from the case of Kähler manifolds to that of algebraic varieties, we must study
Hodge theory and use the older results of Deligne \cite{Deligne} and Morgan \cite{Morgan},
reviewed in section~\ref{SecHodgeRationalHomotopy}.
We restrict ourself to the case of quasi-projective varieties.
In section~\ref{SecMainTheorem}
we review the ideas of Kapovich and Millson in \cite{KapovichMillson} and
our main result is:

\begin{The}
Let $X$ be a smooth complex quasi-projective variety and $\Gamma$ its fundamental group.
Let $\rho:\Gamma\rightarrow G(\RR)$ be a representation with finite image.
Corresponding to $\Ker(\rho)$ there is a finite étale Galois cover ${Y\rightarrow X}$.
Assume that $Y$ has a smooth compactification $\overline{Y}$ with first Betti number $b_1(\overline{Y})=0$.
Then $(\Hom(\Gamma, G), \rho)$ is a quadratic singularity.
\end{The}

Recall what was proven in \cite{KapovichMillson}.

\begin{The}[Kapovich-Millson]
Let $X$ be a smooth complex algebraic variety and $\Gamma$ its fundamental group.
Let $\rho:\Gamma\rightarrow G(\RR)$ be a representation with finite image.
Then $(\Hom(\Gamma, G), \rho)$ is a weighted homogeneous singularity with generators
of weight $1,2$ and relations of weight $2,3,4$.
\end{The}
 
The case where $\rho$ is the trivial representation, and the assumption
is that $X$ itself has a compactification $\overline{X}$ with $b_1(\overline{X})=0$, is already interesting but known:
many examples come from complements of arrangements of hyperplanes in complex projective space,
or more generally complements of projective algebraic curves. The result
can be obtained by combining the study of the trivial representation in \cite[section 17]{KapovichMillson}
with the notion of $1$-formality developped in the work of Dimca, Papadima and Suciu,
see \cite{DimcaPapadimaSuciuFormality}, \cite{papadimaSuciuFormality}.

In section~\ref{SecExamples} we look for new examples to apply our theorem.
We are first motivated by the case of arrangements and some special representations, 
and we study other classes of examples where we can apply our criterion with respect
to \emph{all} representations with finite image.
We find examples coming from families of complex tori in~\ref{SecAbelianVarieties}
and hermitian locally symmetric spaces in~\ref{SecSymmetricSpaces}.

\section{Preliminaries}

\subsection{Review of Goldman-Millson theory}
\label{SecGoldmanMillson}

We first give a review of \cite{GoldmanMillson}.
We fix a field $k$ of characteristic zero, usually $k=\RR$ or $\CC$.
Our schemes will always be of finite type over $k$.

\subsubsection{Representation variety}
\label{SecRepresentationVariety}
Let $\Gamma$ be a finitely presented group and let $G$ be a linear algebraic group over $k$.
There exists an affine scheme over $k$, denoted by $\Repr$, 
that represents the functor from $k$-algebras to sets
\[ \begin{split} A&\longmapsto\Hom(\Gamma, G(A)) \\
    (f:A\rightarrow B)&\longmapsto f_*:\Hom(\Gamma, G(A))\rightarrow\Hom(\Gamma, G(B)). \end{split} \]
It is called the \emph{representation variety}
(it is in general not a variety, but a scheme).
Thus, giving a representation $\rho:\Gamma\rightarrow G(k)$
is the same as giving a $k$-point of $\Repr$.
When doing topology we just write $G$ for $G(\RR)$.

\subsubsection{Analytic germs}
\label{SecAnalyticGerms}
Given a scheme $S$ and a $k$-rational point $s$, 
the isomorphism class of the complete local ring
$\widehat{O}_{S,s}$ is referred to as the \emph{analytic germ}
of $S$ at $s$. That is, two germs $(S_1,s_1)$ and $(S_2,s_2)$
are said to be \emph{analytically isomorphic} if their complete local rings are isomorphic.

A \emph{weighted homogeneous cone} is an affine scheme defined by equations
of the form $P_j(X_1,\dots,X_n)=0$ in $k^n$ where the variables $X_i$ have given \emph{weights}
$w_i$ and the polynomials $P_j$ are homogeneous of degree $d_j>0$ with respect to the weights $w_i$
(a monomial $X_1^{\alpha_1}\dots X_n^{\alpha_n}$ is of weighted degree $w_1 \alpha_1 + \cdots + w_n \alpha_n$).
We say that $w_i$ are the \emph{weights of the generators} and $d_j$ are the
\emph{weights of the relations}.
The cone is said to be \emph{quadratic} if the $X_i$ have weight $1$ and all the relations have weight $2$.
We say that an analytic germ $(S,s)$ is a \emph{weighted homogeneous singularity} with given weights
(for example, is a quadratic singularity)
if it is analytically isomorphic to a weighted homogeneous cone with these given weights.

\begin{Lem}
\label{LemAnalyticGermRC}
A germe $(S,s)$ is a weighted homogeneous singularity over $\RR$
if and only if it is over $\CC$, with the same weights.
\end{Lem}

\begin{proof}
Of course a weighted homogeneous cone over $\RR$ can be complexified to a cone over $\CC$
with the same weights. In the other direction, given the equations
$P_j(X_1,\dots,X_n)$ over $\CC$, replace the variables $X_i$ by
their real and imaginary parts $x_i, y_i$
and give the same weight $w_i$ to the two new variables. Then
expand the relations $P_j(x_1+iy_1,\dots,x_n+iy_n)=0$, 
separate real and imaginary part, and this gives two equations both with
the same weighted homogeneous degree $d_j$.
\end{proof}

We denote by $\kArt$ the category of Artin local $k$-algebras.
An element $A$ in $\kArt$
has a unique maximal ideal which we always denote by $\mathfrak{m}$, has residue field $k$
and is of finite dimension over $k$. This implies that $\mathfrak{m}$ is a nilpotent ideal, and this
gives a natural map $A\rightarrow k$ which is reduction modulo $\mathfrak{m}$.
An analytic germ $(S,s)$ defines a functor of Artin rings 
\[ \begin{split} F_{S,s}:\kArt & \longrightarrow\mathbf{Set} \\
 A & \longmapsto\Hom(\widehat{O}_{S,s}, A). \end{split} \]
Such a functor is called \emph{pro-representable}.

\begin{The}[{\cite[Theorem~3.1]{GoldmanMillson}}]
Two germs $(S_1,s_1)$ and $(S_2,s_2)$ are analytically isomorphic
if and only if
the associated pro-representable functors
$F_{S_1,s_1}$, $F_{S_2,s_2}$ are isomorphic.
\end{The}

Thus, in order to study the analytic germ of a representation $\rho$
in the representation variety we only have to study its pro-representable functor,
which is also the functor
\[ \begin{split} \kArt & \longrightarrow\mathbf{Set} \\
A & \longmapsto \{ \tilde{\rho}\in\Hom(\Gamma, G(A))\ |\ \tilde{\rho}=\rho \mod\mathfrak{m} \} \end{split} \]
interpreted as the functor of deformations of $\rho$ over $\kArt$;
and the type of analytic singularity corresponds to the obstruction theory for deformations
of $\rho$.

\subsubsection{Differential graded Lie algebras}

Let $L$ be a differential graded Lie algebra.
It has a grading $L=\oplus_{i\geq 0} L^i$, a bracket $[-,-]$
with $[L^i,L^j]\subset L^{i+j}$,
and a derivation $d$ of degree $1$ satisfying the usual identities in the graded sense,
see \cite{GoldmanMillson} and see also \cite{Manetti}.
The basic example is: take a differential graded algebra $\mathcal{A}$ commutative in the
graded sense (for example the De Rham algebra of a smooth manifold)
and a Lie algebra $\mathfrak{g}$ and consider the tensor product $\mathcal{A}\otimes\mathfrak{g}$
with bracket 
\[ [\alpha\otimes u, \beta \otimes v]:=(\alpha\wedge\beta)\otimes[u,v] \]
and differential
\[ d(\alpha\otimes u) := (d\alpha)\otimes u . \]

For such an $L$, $L^0$ is a usual Lie algebra
and for and Artin local $k$-algebra $A$,
$L^0\otimes\mathfrak{m}$ is a nilpotent Lie algebra on wich we can define a group structure via the
Baker-Campbell-Hausdorff formula. This group is denoted by $\exp(L^0\otimes\mathfrak{m})$.

Recall that a \emph{groupoid} is a small category $\mathscr{C}$ in which all arrows are invertible.
An example is provided by the so-called \emph{action groupoid}: let a group $H$ act on a set $E$, take the set of objects
to be $\Obj \mathscr{C}:=E$ and the arrows $x\rightarrow y$ are the $h\in H$ such that $h.x=y$.
We denote by $\Iso \mathscr{C}$ the set of isomorphism classes of $\mathscr{C}$;
in the case of an action groupoid this is just the usual quotient $E/H$.

We define a functor $A\mapsto \mathscr{C}(L,A)$
from $\kArt$ to groupoids, called the \emph{Deligne-Goldman-Millson functor}: 
the set of objects is
\begin{equation}
\label{EqMaurerCartan}
\Obj \mathscr{C}(L,A) := \{\eta\in L^1 \otimes \mathfrak{m} \ |\ d\eta+\frac{1}{2}[\eta,\eta]=0 \}
\end{equation}
(this one is called the \emph{Maurer-Cartan equation})
and the arrows of the groupoid are given by the action of the group
$\exp(L^0\otimes\mathfrak{m})$
by
\[ \exp(\alpha).\eta := \eta + \sum_{n=0}^{\infty} \frac{(\ad(\alpha))^n}{(n+1)!}([\alpha,\eta]-d\alpha). \]
We then have a functor $A\mapsto\Iso\mathscr{C}(L,A)$ from $\kArt$ to sets.

Given a scheme $S$ over $k$ and a $k$-rational point $s$, we say that
$L$ \emph{controls} the germ $(S,s)$ if the functor $A\mapsto\mathscr{C}(L,A)$
is isomorphic to the functor $F_{S,s}$.

Recall that if $L$ and $L'$ are two differential graded Lie algebras,
a morphism $\phi:L\rightarrow L'$ is said to be a \emph{$1$-quasi-isomorphism}
if it induces an isomorphism on the cohomology groups $H^0$, $H^1$ and an injection on $H^2$.
The algebras $L$, $L'$ are said to be \emph{$1$-quasi-isomorphic}
if there is a sequence of $1$-quasi-isomorphisms connecting them.
 
\begin{The}[{\cite[Theorem~2.4]{GoldmanMillson}}]
\label{TheGoldmanMillsonQuasiIsomorphism}
If $L$ and $L'$ are two differential graded lie algebras and $\phi:L\rightarrow L'$
is a $1$-quasi-isomorphism,
then the germs controlled by $L$ and $L'$ are analytically isomorphic.
\end{The}

Thus to understand an analytic germ $(S,s)$ it suffices to understand the functor
$A\mapsto\Iso\mathscr{C}(L,A)$ for some controlling algebra $L$ up to $1$-quasi-isomorphism.

\begin{Rem}
\label{RemL0}
If $L^0=0$ then this functor is equal to $A\mapsto\Obj\mathscr{C}(L,A)$.
If futhermore $L^1$ is finite-dimensional this is  exactly the pro-representable functor
associated to the analytic germ at $0$ of the Maurer-Cartan equation
$d\eta+\frac{1}{2}[\eta,\eta]=0 $ for $\eta\in L^1$.
\end{Rem}

\subsubsection{Main construction}
\label{subsubMainObjects}
We explain the main construction to relate theses objects.
Let $X$ be a real manifold, $x$ a base point, $\Gamma$ its fundamental group and $G$ a linear algebraic group over $\RR$.
Let $\gg$ be the Lie algebra of $G$. Let $\rho:\Gamma\rightarrow G(\RR)$ a representation.
Let $P$ be the principal bundle obtained by the left monodromy action of $\Gamma$
on $G$ via $\rho$. If $\widetilde{X}$ is a universal covering space for $X$, on which 
we make $\Gamma$ act on the left, then 
\[ P:=\widetilde{X}\times_\Gamma G=(\widetilde{X}\times G)/\Gamma \]
where $\Gamma$ acts diagonally.
The group $G$ acts on its Lie algebra via the adjoint representation $\Ad$
and $\Gamma$ acts on $\gg$ by $\Ad\circ\rho$.
We denote by $\Ad(P)$ the adjoint bundle
\[ \Ad(P):=P\times_G \gg =\widetilde{X}\times_\Gamma \gg \]
and it comes with a flat connection
such that the algebra of differential forms with value in $\Ad(P)$, denoted by
$\mathcal{A}^\bullet(X, \Ad(P))$, has the structure of a differential graded Lie algebra.

Given the base point $x$ we define an augmentation $\epsilon:\mathcal{A}^\bullet(X,\Ad(P))\rightarrow\gg$
by evaluating degree $0$ forms at $x$ and sending the others to zero.
We put $\mathcal{A}^\bullet(X,\Ad(P))_0:=\Ker(\epsilon)$.

\begin{The}[{\cite[Theorem~6.8]{GoldmanMillson}}]
The differential graded Lie algebra $\mathcal{A}^\bullet(X,\Ad(P))_0$
controls the analytic germ $(\Repr, \rho)$.
\end{The}

It is then proven in \cite{GoldmanMillson} that when $X$ is a compact Kähler manifold,
$\mathcal{A}^\bullet(X,\Ad(P))_0$ is quasi-isomorphic (over $\CC$)
to a differential graded Lie algebra with zero differential
and this implies that the analytic germ controlled by it is quadratic.

\subsection{Mixed Hodge theory and rational homotopy}
\label{SecHodgeRationalHomotopy}
Next we give a short review of
\cite{Deligne} and \cite{Morgan}.
See also \cite{PetersSteenbrink}.

\subsubsection{Hodge structures}
Given a finite-dimensional vector space $V$ over $\RR$, a \emph{Hodge structure}
of weight $n$ over $V$ is the data of a decreasing (finite) filtration $F$ of $V_\CC := V\otimes \CC$
and a decomposition in bigraded parts $V_\CC = \bigoplus_{i+j=n} V^{i,j}$, where
$F^p(V_\CC) = \bigoplus_{i\geq p} V^{i,j}$,
with $\overline{V^{j,i}} = V^{i,j}$.
A \emph{mixed Hodge structure} on $V$ is the data of an increasing (finite) filtration $W$, a decreasing filtration $F$
of $V_\CC$, such that $F$ induces on $\Gr_n^W (V)$ a Hodge structure of weight $n$, for all $n$;
so $\Gr_n^W(V_\CC)=\bigoplus_{i+j=n} V^{i,j}$ with $\overline{V^{j,i}}=V^{i,j}$ modulo $W_{n-1}(V_\CC)$.
The category of mixed Hodge structures is abelian (this is not trivial, see \cite[Theorem~1.2.10]{Deligne}).

Given a mixed Hodge structure on $V$, there is one preferred way of splitting $V_\CC = \bigoplus A^{i,j}$
such that $W_n (V_\CC) = \bigoplus_{i+j\leq n} A^{i,j}$, $F^p(V_\CC)=\bigoplus_{i\geq p} A^{i,j}$
and then $V^{i,j}$ becomes naturally isomorphic with $A^{i,j}$.
We call it the \emph{Deligne splitting}, see  \cite[Proposition~1.9]{Morgan}

\subsubsection{Hodge structures on cohomology groups}
Suppose that $X$ is a smooth complex quasi-projective variety. We denote by $\overline{X}$
a smooth compactification such that $D:=\overline{X} \setminus X$ is a divisor with normal crossings.
By Deligne \cite{Deligne}, the cohomology groups of $X$ are equipped with a mixed Hodge structure which is
independant of the choice of $\overline{X}$. On $H^n(X)$ the nonzero graded parts for $W$ are of weight between $n$ and $2n$
and $\Gr_n^W H^n(X) = H^n(\overline{X})$.

In our case of interest, we will have a finite étale Galois cover $Y\rightarrow X$
and we are interested in the condition $b_1(\overline{Y})=0$.
Remark that since $b_1$ is a birational invariant, it is enough to have for $\overline{Y}$ a smooth compactification,
not necessarily by a divisor with normal crossings.

The condition $b_1(\overline{Y})=0$ is equivalent to one of the following:

\begin{itemize}
\item $H^1(\overline{Y}, \QQ)=0$.
\item $\pi_1(\overline{Y})^{\mathrm{ab}}$ if finite.
\item $q(\overline{Y})=0$, where $q:=\dim H^0(\overline{Y}, \Omega^1_{\overline{Y}})$ is the irregularity.
\item The mixed Hodge structure on $H^1(Y)$ is pure of weight $2$.
\end{itemize}

There is also one characterization that will motivate our section
\ref{SecAbelianVarieties}.

\begin{Lem}
\label{LemCriteriaAlbanese}
Let $X$ be a smooth quasi-projective variety. Then $b_1(\overline{X})=0$ if and only if there is no
nonconstant holomorphic map $f:X\rightarrow A$ to an abelian variety $A$.
\end{Lem}

\begin{proof}
Recall that $\overline{X}$ has an Albanese variety, which is an abelian variety
$\operatorname{Alb}$ with a map $\operatorname{alb}:\overline{X}\rightarrow \operatorname{Alb}$ 
through which every map to an abelian variety factors. Its dimension
is exactly $h^{0,1}(\overline{X})=b_1(\overline{X})/2$.

If $b_1(\overline{X})>0$ then $\operatorname{alb}$ restricts to a map $f:X\rightarrow\operatorname{Alb}$.
Conversely given a nonconstant map $f:X\rightarrow A$ then first $f$ extends to $\overline{X}$
(see \cite[Theorem~9.4 chapter~4]{LangeBirkenHake}), then factors through
$\operatorname{alb}$ and so $\operatorname{Alb}$ must be of positive dimension.
\end{proof}

\subsubsection{Rational homotopy theory}
We explain very briefly the ideas we need. We refer to \cite{GriffithsMorgan}
and \cite{Morgan}.

Let $A^\bullet$ be a (commutative) differential graded algebra over a field $k=\RR$ or $\CC$: 
this means that $A$ has a grading $A=\oplus_{i\geq 0} A^i$, a
multiplication $\wedge:A^i\otimes A^j\rightarrow A^{i+j}$ 
with $\alpha\wedge\beta=(-1)^{ij}\beta\wedge\alpha$ if $\alpha\in A^i$,$\beta\in A^j$,
and a derivation of degree $1$ with $d(\alpha\wedge\beta)=d\alpha\wedge\beta+(-1)^i \alpha\wedge d\beta$.
We denote by $A^{+}:=\bigoplus_{i>0} A^i$.

There is a notion of \emph{$1$-minimal algebra} but we won't need the details.
Briefly it means that $A$ is obtained as an increasing union of elementary extensions, with
in addition $A^0=k$ and $d$ is \emph{decomposable}
which means $d(A)\subset A^{+}\wedge A^{+}$.

A \emph{$1$-minimal model} for $A$ is a $1$-minimal differential algebra $\mathscr{N}$ with a morphism
$\nu:\mathscr{N}\rightarrow A$ which is a $1$-quasi-isomorphism. See \cite{GriffithsMorgan}
and \cite{Morgan} for more statements about existence and unicity.
Remark that if $\nu:\mathscr{N}\rightarrow A$ is a $1$-minimal model over $\RR$
then $\nu_\CC:\mathscr{N}_\CC\rightarrow A_\CC$ is a $1$-minimal model over $\CC$.

\subsubsection{Hodge structures on minimal models}
We recall the work of Morgan \cite{Morgan}. The goal is to put mixed Hodge structures on
several rational homotopy invariants of algebraic varieties, however we only need to study the $1$-minimal model.

Let $X$ be a smooth complex quasi-projective variety admitting a compactification $\overline{X}$ by a divisor with normal crossings $D$.
We denote as always by $\aab(X, \RR)$ the algebra of real-valued differential forms on $X$,
and by $\mathscr{E}^\bullet(\log D)$ the algebra of $\mathcal{C}^\infty$ complex-valued forms on $X$
with logarithmic poles along $D$, with its filtration $W$ by the order of poles. There is a canonical map
$\mathscr{E}^\bullet(\log D)\hookrightarrow\aab(X, \CC)$ that is a quasi-isomorphism.

In addition, Morgan constructs a real algebra $E^\bullet_{\mathcal{C}^\infty}(X)$, see \cite[Section~2]{Morgan} for details,
with a filtration $W$ similar to the weight filtration on $\mathscr{E}(\log D)$, and 
constructs a
quasi-isomorphism $E_{\mathcal{C}^\infty}(X)\otimes\CC \rightarrow \mathscr{E}(\log D)$ that respects $W$.
Now let $A:=E_{\mathcal{C}^\infty}(X)$. Recall the Dec-weight filtration defined by
\[ \Dec W_i (A^n) = \{x\in W_{i-n}(A) \ |\ dx\in W_{i-n-1}(A^{n+1}) \} \]
such that for the spectral sequence
\[{}_W E_r^{p,q}= {}_{\Dec W}E_{r-1}^{-q,p+2q}\]
and on cohomology
\[ \Dec W_i H^n(A)=W_{i-n} H^n(A) .\]
Morgan proves that $A$ has a real minimal model
$\nu:\mathscr{N}\rightarrow A$ with a filtration $W$ such that $\nu$ respects the filtrations
$(\mathscr{N},W)\rightarrow (A,\Dec W)$; of course by transitivity $(\mathscr{N}_\CC,W)$
is a also a filtered (complex) minimal model for $(\mathscr{E}(\log D),W)$.

Over $\CC$, the weight filtration on $\mathscr{N}_\CC$ splits and
we will denote by lower indices the grading by the weight, that is compatible
with the grading of $\mathscr{N}$ as differential graded algebra. So $\mathscr{N}_\CC=\bigoplus_{i\geq 0}\mathscr{N}_i$, with
$d:\mathscr{N}_i \rightarrow \mathscr{N}_i$ and $\mathscr{N}_i \wedge\mathscr{N}_j\subset\mathscr{N}_{i+j}$.
The grading has the following properties:
\begin{itemize}
 \item Via the $1$-quasi-isomorphism $\nu$, the grading by weight induced on $H^n(\mathscr{N}_\CC)$
coincide with the grading induced on $H^n(X, \CC)$ by the Deligne splitting of its mixed Hodge structure.
\item For $n=1$ the only possible weights induced are $1,2$ and for $n=2$ they are $2,3,4$.
\item Each component $\mathscr{N}_i^j$ is of finite dimension.
\item $\mathscr{N}_0=\CC$ and $\mathscr{N}^0=\CC$.
\item $d(\mathscr{N}_1^1)=0$
(if $x\in\mathscr{N}_1^1$ then as $d$ is decomposable, $dx=\sum \alpha_i\wedge\beta_i\in\mathscr{N}_1^2$, so for each
$i$ we must have $\alpha_i\in\mathscr{N}_1^1$ and $\beta_i\in\mathscr{N}_0^1$ or the contrary; but
$\mathscr{N}_0=\CC$ concentrated in degree $0$ so $\mathscr{N}_0^1=0$).
\end{itemize}

\section{Equivariant constructions and proof of the main theorem}
\label{SectionEquivariantConstruction}
Now we rewrite the section \ref{subsubMainObjects} taking into account
a finite covering space and equivariance, as needed in \cite[section~15]{KapovichMillson}.
From now on we fix the objects we introduce: 
$X$ is a smooth complex quasi-projective variety 
with a base point $x$
and fundamental group $\Gamma$. We fix a linear algebraic group $G$
over $\RR$ with Lie algebra $\gg$ and a representation $\rho:\Gamma\rightarrow G(\RR)$
with finite image. We introduce the finite group $\Phi:=\Gamma/\Ker(\rho)\simeq\rho(\Gamma)$.

\subsection{Covering spaces}

\subsubsection{Covering spaces and compactifications}

To $\Ker(\rho)$ corresponds a finite étale Galois cover $\pi:Y\rightarrow X$
with automorphism group $\Phi$ that acts simply transitively.
Automatically $Y$ is a smooth quasi-projective variety.

It is possible to choose compactifications $\overline{Y},\overline{X}$, both smooth by divisors with normal crossings,
in an equivariant way. This means that $\pi$ extends to a finite ramified cover $\overline{Y}\rightarrow \overline{X}$
over $\overline{X}\setminus X$.
\[ \xymatrix{
\Phi\curvearrowright  Y \ar@{^{(}->}[r] \ar[d]_\pi & \overline{Y} \ar[d]^\pi  \curvearrowleft\Phi  \\
X \ar@{^{(}->}[r] & \overline{X}  } \]
The action of $\Phi$ extends to $\overline{Y}$ with $\pi$ being invariant, so
leaves globally invariant $\overline{Y}\setminus Y$; and $Y$ comes equipped with a base point $y$ over $x$.

\subsubsection{Bundles and augmentations}
We construct the bundle $\Ad(P)$ taking into account the augmentation.

First fix some notations: if $E$ is a flat bundle, there is 
a twisted algebra of differential forms with value in $E$ denoted by $\aab(X, E)$.
If $E$ is globally trivial this is just $\aab(X)\otimes E$ where we write $\aab(X)$ for $\aab(X,\RR)$.
Given a group $\Phi$ acting on an algebra $A$ (which can be graded, commutative, Lie\ldots
and the action must respect this structure)
we always denote by $A^\Phi$ the sub-algebra of invariants by $\Phi$.
Given an augmentation of $A$, we always denote by $A_0$ the kernel of the augmentation.

Introduce the trivial bundle $Q:=Y \times G$ and its adjoint bundle
$\Ad(Q):=Y\times\gg$.
Recall that $\Phi$ acts on $Y$ with $Y/\Phi=X$; on $G$ (via $\rho$)
and on $\gg$ (via $\Ad\circ\rho$); and also on $\aab(Y)$
with $\aab(Y)^\Phi=\aab(X)$. It also acts naturally on products and tensor products.

So: we have $P=Q/\Phi$ and $\Ad(P)=\Ad(Q)/\Phi$, and for the twisted version
\[ \aab(X, \Ad(P)) = (\aab(Y, \Ad(P')))^\Phi = (\aab(Y)\otimes\gg)^\Phi. \]

We want to lift the augmentation
$\beta:\aab(X)\rightarrow\RR$, which is the evaluation of $0$-forms at $x$, to $Y$. We let
\[ \beta_Y(f) := \frac{1}{|\Phi|} \sum_{\gamma\in\Phi} (\gamma.f)(y) \]
for $f\in\mathcal{A}^0(Y)$.
Thus we sum over all of $\pi^{-1}(x)$.
Then naturally 
\[ \aab(X)_0 = (\aab(Y)^\Phi)_0 = (\aab(Y)_0)^\Phi \]
(the first two augmentations are with respect to $\beta$, the last one to $\beta_Y$)
and we can write it $\aab(Y)_0^\Phi$.

In the same way we want to lift $\epsilon:\aab(X, \Ad(P))\rightarrow \gg$
to $\aab(Y)\otimes\gg$.
Put 
\[ \epsilon_Y(f\otimes u):=\frac{1}{|\Phi|} \sum_{\gamma\in\Phi}
(\gamma.(f\otimes u))(y). \]
Then naturally
\begin{equation}
\label{EqEquivarianceAugmentation}
\aab(X,\Ad(P))_0 = ((\aab(Y)\otimes\gg)^\Phi)_0
 = ((\aab(Y)\otimes\gg)_0)^\Phi = (\aab(Y)_0 \otimes\gg)^\Phi
\end{equation}
and we can denote all this by $(\aab(Y)\otimes\gg)_0^\Phi$.

Observe that all theses constructions extend naturally to $\CC$.

\begin{Lem}[See {\cite[Lemma 5.12]{GoldmanMillson}}]
\label{LemSurjectiviteAugmentations}
The four augmentations defined above are surjective.
This implies that $H^0(\aab(X,\RR)_0)=0$ and $H^0(\aab(X, \Ad(P))_0)=0$.
\end{Lem}

\subsubsection{Cohomology}

We give a usefull lemma concerning the relation between cohomology and invariants.

\begin{Lem}
Let $A^\bullet$ be a differential graded commutative algebra (or Lie algebra) over a field of characteristic zero.
Let $\Phi$ be a finite group acting on $A^\bullet$. Then on cohomology
\[ H^\bullet(A^\Phi) = (H^\bullet(A))^\Phi. \]
\end{Lem}
\begin{proof}
An element in $(H^i(A))^\Phi$ is given by an element $x\in A^i$ with $dx=0$ such that for all
$\gamma\in\Phi$ there exists an element $\alpha_\gamma\in A^{i-1}$ with
$\gamma.x = x+d(\alpha_\gamma)$.
A cohomology class in $H^i(A^\Phi)$ is given by an element $z\in A^i$
which is $\Phi$-invariant and $dz=0$.
So there is a natural map $H^i(A^\Phi)\rightarrow (H^i(A))^\Phi$
taking $z\in H^i(A^\Phi)$ to an element in $(H^i(A))^\Phi$
with $\alpha_\gamma=0$ for all $\gamma$.
In the other direction, given $x \in (H^i(A))^\Phi$ we put
\[ z := \frac{1}{|\Phi|} \sum_{\gamma\in\Phi} \gamma.x = x+d\left(\sum_{\gamma\in\Phi}\alpha_\gamma\right) \]
so that $dz=0$ and $z$ is $\Phi$-invariant so is in $H^i(A^\Phi)$ and induces $x$.

Now this proves easily that the natural map is an isomorphism.
\end{proof}

\begin{Cor}
\label{CorQuasiIsomorphismEquivariant}
Let $L,L'$ be differential graded Lie algebras, $\psi:L\rightarrow L'$ a $1$-quasi-isomorphism.
Suppose that a finite group $\Phi$ acts on $L$, $L'$ and $\psi$ commutes with the actions.
Then $\psi$ induces a $1$-quasi-isomorphism $\psi^\Phi : L^\Phi\rightarrow L'^\Phi$.
\end{Cor}

\subsubsection{Equivariant minimal model}
We refer to \cite[section~15]{KapovichMillson} for this technical part.
We denote by $A^\bullet$ the algebra $E^\bullet_{\mathcal{C}^\infty}(Y)$ with its filtration $W$.
It is shown that $\Phi$ acts on $(A^\bullet, W)$
and $(A^\bullet)^\Phi$ is then quasi-isomorphic to $E^\bullet_{\mathcal{C}^\infty}(X)$.
It is shown how to construct a $1$-minimal model $\nu:\mathscr{N}\rightarrow A$ with a filtration $W$
and an action of $\Phi$ which commutes with $\nu$; and $\nu$
respects the filtration $\Dec W$ on $A$. So
(corollary \ref{CorQuasiIsomorphismEquivariant})
$\mathscr{N}^\Phi$ is a $1$-minimal model for $A^\Phi=E^\bullet_{\mathcal{C}^\infty}(X)$.

Also, by transitivity and over $\CC$, $\mathscr{N}_\CC$ is a $1$-minimal model for $\mathcal{A}(Y,\CC)$
and $\mathscr{N}_\CC^\Phi$ is a $1$-minimal model for $\mathcal{A}(X, \CC)$.
The filtration $W$ on $\mathscr{N}_\CC$ splits and $\Phi$ commutes with the splitting.

\subsection{Minimal model for a Lie algebra}
\label{SecMainTheorem}

We are now able to describe an explicit differential graded Lie algebra
which controls the germ of $\rho$ in $\Repr$, and which comes with a grading by weight.
Everyting is done in \cite[section 15]{KapovichMillson}.

\subsubsection{The minimal model}
Consider $\mathscr{G}:=\mathscr{N}\otimes\gg$ and put
$\nu\otimes\id:\mathscr{G}\rightarrow A\otimes\gg$.
Then $\mathscr{G}$ is a differential graded Lie algebra with a differential $d$ decomposable,
and $\nu\otimes\id$ is a $1$-quasi-isomorphism.
We can call $\mathscr{G}$ a \emph{$1$-minimal model in the sense of Lie algebras}.

The grading by weight on 
$\mathscr{N}_\CC$ induces one on $\mathscr{G}_\CC$ with the properties that:
\begin{itemize}
\item $\mathscr{G}_\CC = \bigoplus_{i\geq 0} \mathscr{G}_i$.
\item Each $\mathscr{G}_i^j$ is finite dimensional.
\item $[\mathscr{G}_i, \mathscr{G}_j]\subset\mathscr{G}_{i+j}$.
\item $d(\mathscr{G}_i)\subset\mathscr{G}_i$, so the cohomology is also graded.
\item $\mathscr{G}_0=\gg_\CC$ and $\mathscr{G}^0_\CC=\gg_\CC$.
\item The only non-zero induced weights on $H^n(\mathscr{G}_\CC)$
are $1,2$ for $n=1$ and $2,3,4$ for $n=2$.
\end{itemize}
The group $\Phi$ acts on both factors of $\mathscr{G}$ and preserve the graduation on $\mathscr{G}_\CC$.
Put $\mathscr{M}:=\mathscr{G}^\Phi$, so that $\mathscr{M}_\CC$ has a bigrading with the same properties;
by transitivity
$\mathscr{M}_\CC$ is $1$-quasi-isomorphic to $\aab(X, \Ad(P)_\CC)$.

\subsubsection{Augmentation}
Recall the augmentations $\beta,\epsilon,\beta_Y,\epsilon_Y$, extend them over $\CC$.
It is easy to pull them back respectively to $A_\CC$, $A_\CC \otimes\gg$,
$(A_\CC)^\Phi$, $(A_\CC\otimes\gg)^\Phi$ and then to
$\mathscr{N}_\CC$, $\mathscr{G}_\CC$, $\mathscr{N}_\CC^\Phi$,
$\mathscr{G}_\CC^\Phi$ all in a compatible way:
$\mathscr{N}_{\CC,0}$
(warning with the notations, this is the kernel of the augmentation
and $\mathscr{N}_0$ is the degree zero graded part by weight on $\mathscr{N}_\CC$)
is then $1$-quasi-isomorphic to $\mathcal{A}(Y,\CC)_0$
such that $ (\mathscr{N}_{\CC,0})^\Phi = (\mathscr{N}_\CC^\Phi)_0 $
and $\mathscr{G}_{\CC,0}$ is $1$-quasi-isomorphic to $\mathcal{A}(X,\Ad(P)_\CC)_0$
such that $ (\mathscr{G}_{\CC,0})^\Phi = (\mathscr{G}_\CC^\Phi)_0 $;
and we denote this last one by $\mathscr{L}$.

Recall that when constructing a minimal model $\mathscr{N}\rightarrow A$, $\mathscr{N}^0$ is sent isomorphically
to $H^0(A)$.
Combining this with lemma~\ref{LemSurjectiviteAugmentations}, and with the various compatibilities of the augmentations, we
see that $\mathscr{L}^0=0$. Furthermore $\mathscr{L}_0=0$.

\subsection{The controlling algebra and proof of the main theorem}

\subsubsection{The controlling algebra}

By the theorem~\ref{TheGoldmanMillsonQuasiIsomorphism}
and the remark~\ref{RemL0} the analytic germ of $\rho$ in $\Repr$
is isomorphic (over $\CC$) to the one at $0$ in $\mathscr{L}^1$ of the equation $d\eta+\frac{1}{2}[\eta,\eta]=0$.
So $\mathscr{L}$ is the controlling algebra to our problem.

We can simplify it as in \cite[section 15]{KapovichMillson}.
Put $\mathscr{I}:=\mathscr{L}_4^1 \oplus d(\mathscr{L}_4^1) \oplus\bigoplus_{i\geq5}\mathscr{L}_i$, observe
that it is and ideal in $\mathscr{L}$ (in the sense of differential graded Lie algebras
with an additional grading by weight, that is:
bigraded homogeneous, stable by $d$, and stable by Lie bracket with $\mathscr{L}$)
such that the projection $\mathscr{L}\rightarrow Q:=\mathscr{L}/\mathscr{I}$ is
a $1$-quasi-isomorphism.
This $Q$ is simpler to study because it has all the properties of $\mathscr{G}_\CC$ above and in addition
$Q_i=0$ for $i\geq 5$ and $Q_4^1=0$.

\subsubsection{Proof of the main theorem}

\begin{The}
If we assume that the compactification $\overline{Y}$ has $b_1(\overline{Y})$=0, then
the algebra $Q$ above controls a quadratic germ.
\end{The}

\begin{proof}
By hypothesis $H^1(Y)$ is a pure Hodge structure of weight $2$.
Looking carefully at our $1$-quasi-isomorphisms
that preserve filtrations by weight and graduations over $\CC$, we see
that on $H^1(\mathscr{N}_i)$ the only nonzero induced weight is for $i=2$.
This special property is also true for $\mathscr{G}_\CC$ because
\[ H^\bullet(\mathscr{G}_i)=H^\bullet(\mathscr{N}_i)\otimes\gg \]
and also for $\mathscr{G}_{\CC,0}$ because
\[ H^\bullet(\mathscr{G}_{i,0})=H^\bullet((\mathscr{N}_i\otimes\gg)_0)
 = H^\bullet(\mathscr{N}_{i,0}\otimes\gg) = H^\bullet(\mathscr{N}_{i,0})\otimes\gg \]
and $\mathscr{N}_{\CC,0}$ is a subcomplex of $\mathscr{N}_\CC$, so all restrictions apply to it.

This restriction on weights holds for $\mathscr{L}$ because the action of $\Phi$ preserves the graduation and 
\[ H^\bullet(\mathscr{L}_i) = H^\bullet((\mathscr{G}_{i,0})^\Phi) = (H^\bullet(\mathscr{G}_{\CC,0}))^\Phi , \]
and by the $1$-quasi-isomorphism it holds for $Q$.

Now look at the equation $d\eta+\frac{1}{2}[\eta,\eta]$ in $Q^2$, for $\eta\in Q^1$.
By construction $Q^1=Q^1_1 \oplus Q^1_2 \oplus Q^1_3$ (and $Q^1_0=0$).
By our hypothesis $H^1(Q_1)=0$, which is
$\Ker(d:Q_1^1\rightarrow Q_1^2)/d(Q_1^0)$. Combined with the fact that $d(Q^1_1)=0$
and $Q_0^1=0$ we have $Q_1^1=0$.

So, decompose $\eta=\eta_2+\eta_3$ where $\eta_i$ is of weight $i$. The equation on $\eta$
becomes (we truncate parts of weight $\geq 5$)
\begin{align*}
d\eta_2 & = 0 \\
d\eta_3 & =0 \\
\frac{1}{2}[\eta_2,\eta_2] & = 0.
\end{align*}
Since we have $H^1(Q_3)=0$ and $d\eta_3=0$, $\eta_3$ must be exact. But a primitive 
must be in $Q_3^0$, which is $0$. So we can eliminate the equation $d\eta_3=0$.
Since $d\eta_2=0$, we can just assume $\eta_2$ is in the linear subspace $Z_2^1:=\Ker(d)\cap Q_2^1$
and it remains only the equation
\begin{equation}
\frac{1}{2}[\eta_2,\eta_2] = 0,\quad \eta_2\in Z_2^1
\end{equation}
which is weighted homogeneous, with the generator $\eta_2$ of weight $2$
and the relation of weight $4$.
But we can divide the degrees by two and this is isomorphic to a weighted homogeneous cone
with generators of weight $1$ and relations of weight $2$, that is a quadratic cone. 
\end{proof}

\section{Examples}
\label{SecExamples}

We now investigate several situations where we can apply our main theorem.

\subsection{Abelian coverings of line arrangements}
Motivated by the case of the trivial representation, the first example is to take for
$X$ the complement of an arrangement of hyperplanes in some complex projective space.
We reduce to the case of the projective plane, thus we denote by $\mathcal{L}$
a finite union of lines in $\mathbb{P}^2(\CC)$ and by $X$ its complement.
A smooth compactification of $X$ is obtained as a blow-up of $\mathbb{P}^2(\CC)$
at the points of intersection with multiplicity at least $3$ so has
clearly $b_1(\overline{X})=0$.

There is a special class of coverings of $\mathbb{P}^2(\CC)$ branched over $\mathcal{L}$ that has already been studied:
the \emph{Hirzebruch surfaces}. The definition appears first in
\cite{Hirzebruch},
and a study of the Betti numbers was done in \cite{Hironaka}.
See also \cite{Suciu} for a survey of these results.

For each integer $N$, we define $X_N(\mathcal{L})$ to be the covering of $X$ corresponding to the
morphism $\pi_1(X)\rightarrow H_1(X, \ZZ)\rightarrow H_1(X, \ZZ/N\ZZ)$.
It is known that if $n$ is the number of lines of the arrangement then
$H_1(X, \ZZ)$ is free of rank $n-1$, so $X_N(\mathcal{L})$
is a Galois cover of degree $N^{n-1}$.
It extends to a branched covering $\widehat{X}_N(\mathcal{L})$ over $\mathcal{L}$; 
$\widehat{X}_N(\mathcal{L})$ is a normal algebraic surface. 
We define the Hirzebruch surface associated to $\mathcal{L}$,
which we denote $M_N(\mathcal{L})$, to be a minimal desingularization of $\widehat{X}_N$
(see \cite{Hironaka} for more details).
There are various formulas for computing the Betti number $b_1(M_N(\mathcal{L}))$
and we will refer to Tayama \cite{Tayama}.

\begin{The}[{\cite[Theorem~1.2]{Tayama}}]
Define the function
\[ b(N, n) := (N-1)\left( (n-2)N^{n-2} - 2\sum_{k=0}^{n-3} N^k \right). \]
It is $b_1(M_N(\mathcal{O}_n))$ where $\mathcal{O}_n$
is an arrangement made of $n$ lines passing through a common point.
Let $m_r$ be the number of points of multiplicity $r$ of $\mathcal{L}$.
Let $\beta(\mathcal{L})$ be the number of braid sub-arrangements of $\mathcal{L}$.
Then
\begin{equation}
\label{EqTayamaBound}
b_1(M_N(\mathcal{L})) \geq \sum_{r \geq 3} m_r b(N, r) + \beta(\mathcal{L})b(N, 3).
\end{equation}
Furthermore the following conditions are equivalent:
\begin{itemize}
\item $\mathcal{L}$ is a general position line arrangement.
\item $b_1(M_N(\mathcal{L}))=0$ for any $N\geq 2$.
\item $b_1(M_N(\mathcal{L}))=0$ for some $N\geq 3$.
\end{itemize}
\end{The}

From this we deduce easily one necessary condition for having $b_1(M_N(\mathcal{L}))=0$
and in fact there is a converse.

\begin{The}
We have $b_1(M_N(\mathcal{L}))=0$ if and only if
\begin{itemize}
\item Either $\mathcal{L}$ is a general position line arrangement, and $N$ is any integer.
\item Either $N=2$ and $\mathcal{L}$ has at most triple points.
\end{itemize}
\end{The}

\begin{proof}
The case of general position is already treated in Tayama's theorem.
Now suppose that $\mathcal{L}$ is not in general position and $b_1(M_N(\mathcal{L}))=0$.
Then $N=2$. As $b(3,2)=0$ and $b(3, r) > 0$ if $r >3$ it follows from the inequality
(\ref{EqTayamaBound}) that
we must have $m_r=0$ for $r>3$, that is $\mathcal{L}$ contains at most triple points.

For the converse, it is known (Eyssidieux \cite{Eyssidieux}, private communication)
that if $\mathcal{L}$ has at most triple points and $N=2$ then
$M_N(\mathcal{L})$ is simply connected. So $b_1(M_N(\mathcal{L}))=0$.
\end{proof}

\subsection{Criterion with respect to all finite representations}
As we have seen the case of arrangements is quite limited.
But we have other sources of interesting examples where we can apply our theorem
with respect to all representations with finite image.

\begin{Def}
A smooth complex quasi-projective variety $X$ is said to have property \PropP{} if
for all normal subgroup of finite index $H\subset\pi_1(X)$, the corresponding finite étale Galois
cover $\pi:Y\rightarrow X$ has a (smooth) compactification $\overline{Y}$ with
$b_1(\overline{Y})=0$.
\end{Def}

We will study this property and give two interesting classes of examples.

\subsubsection{Construction of varieties with property \PropP}

It is easy to see that if $X$ is a smooth projective variety with $\pi_1(X)$ finite, then $X$ has property \PropP. Indeed
a finite cover $Y$ a smooth projective variety corresponding to $H\subset\pi_1(H)$ is automatically a smooth projective variety
and recall that $\pi_1(Y)=H$.
Conversely it is known (by Serre) that every finite group is the fundamental group of some smooth projective variety.
It is also clear that if $X$ satisfies \PropP, then any finite étale Galois cover of $X$ satisfies \PropP.

\begin{The}
The variety $\CC^*$ has property \PropP.
\end{The}

\begin{proof}
It is known that every $n$-sheeted cover of $\CC^*$ is of the form
$\CC^*\rightarrow \CC^*$, $z\mapsto z^n$. This can be compactified
in a ramified cover $\mathbb{P}^1\rightarrow \mathbb{P}^1$
over $\{0,\infty\}$ and $b_1(\mathbb{P}^1)=0$.
\end{proof}

So with the next theorem we will easily see that
for every finitely generated abelian group $G$, there is
a smooth quasi-projective variety $X$ with $\pi_1(X)\simeq G$
which has property \PropP.

\begin{The}
\label{TheProduct}
If $X_1$ and $X_2$ satisfy (P), then $X_1 \times X_2$ satisfies (P).
\end{The}

\begin{proof}
Denote by $\Gamma_i:=\pi_1(X_i)$ ($i=1,2$).
Let $H\subset \Gamma_1 \times \Gamma_2$ be a normal subgroup of finite index.
Put $U_i:=H\cap \Gamma_i$.
Then $U_i$ is a normal subgroup of finite index in $\Gamma_i$.
So the Galois cover $Y$ for $H$ lies under the finite Galois cover corrresponding to $U_1\times U_2$,
which is obtained as a product cover $Y_1\times Y_2$ and we have a sequence of finite étale Galois covers
\[ Y_1 \times Y_2 \longrightarrow Y \longrightarrow X_1\times X_2. \]
Taking compactifications this gives a sequence of finite ramified covers
\[ \overline{Y}_1 \times \overline{Y}_2 \longrightarrow 
\overline{Y} \longrightarrow \overline{X}_1\times \overline{X}_2. \]
By property (P) for $Y_i\rightarrow X_i$
we have that $b_1(\overline{Y}_i)=0$ and so $b_1(\overline{Y}_1\times\overline{Y}_2)=0$.
If we had $b_1(\overline{Y})>0$, there would be holomorphic one-forms on $\overline{Y}$ which could be pulled-back injectively
to $\overline{Y}_1\times\overline{Y}_2$, which is not possible.
So $b_1(\overline{Y})=0$ and this proves property (P) for $H$.
\end{proof}

We also have a twisted version, which may allow to construct new examples.

\begin{The}
If $X$ satisfies \PropP, then any (algebraic) principal $\CC^*$-bundle $P$ over $X$ satisfies \PropP.
\end{The}

\begin{proof}
First we show that a finite étale Galois cover of $P$ is a bundle $Q$ with fiber $\CC^*$
over some finite étale Galois cover $\pi:Y\rightarrow X$.
Let $\tau:Q\rightarrow X$ be such a cover, 
with $H:=\tau_*\pi_1(Q)\subset\pi_1(X)$ a normal subgroup of finite index.
Since $p:P\rightarrow X$ is a bundle, the induced map $p_*:\pi_1(P)\rightarrow\pi_1(X)$
is surjective, so $p_*(H)$ is a normal subgroup of finite index in $\pi_1(X)$. This corresponds to
a finite étale Galois cover $\pi:Y\rightarrow X$
with $\pi_*\pi_1(Y)=p_* H$.
\[ \xymatrix{
Q \ar[r]^\tau \ar@{-->}[d]_q & P \ar[d]^{p} \\
Y \ar[r]_{\pi} & X } \]
Now by construction $(p\circ\tau)_*\pi_1(Q)=\pi_*\pi_1(Y)$
which means there is a lifting of $p\circ\tau:Q\rightarrow X$ to a map $q:Q\rightarrow Y$.
All theses spaces and maps can be taken to be algebraic and $Q$ is a bundle
whose general fiber is a finite cover of $\CC^*$, that is $\CC^*$:
we can see this by first looking at $\pi^* P$, which is a principal $\CC^*$-bundle
over $Y$, then the induced map $Q\rightarrow\pi^* P$ over $Y$ which is finite étale.

Now we want to apply lemma~\ref{LemCriteriaAlbanese}.
Let $A$ be an abelian variety and $f:Q\rightarrow A$.
Restricted to each fiber, $f$ is a map $\CC^*\rightarrow A$.
By rigidity it extends to $\mathbb{P}^1(\CC)$, but a map
$\mathbb{P}^1(\CC)\rightarrow A$ is constant.
So $f$ is constant on each fiber and thus is determined by its restriction to $Y$.
But since $X$ has \PropP{} this one is constant so $f$ is globally constant,
which proves $b_1(\overline{Q})=0$.
\end{proof}

\subsubsection{Families of complex tori}
\label{SecAbelianVarieties}
Our motivation for studying complex tori is the use of the Albanese variety as
in lemma~\ref{LemCriteriaAlbanese} and the various rigidity lemmas
for abelian varieties and families.

Let $X$ be a smooth quasi-projective variety.
We would like to prove, using these lemmas, that if $E\rightarrow X$ is a family of abelian varieties over a base
that satisfies \PropP, then $E$ satisfies \PropP.
Of course the family must not contain a constant factor.
However there is some technical difficulty coming from the fact that a finite cover of a family of abelian varieties
may not be a family of abelian varieties, because of the lack of a zero section.
Thus we have to work with complex tori.

\begin{Def}
A \emph{family of complex tori} is a smooth quasi-projective variety $E$
with a smooth projective morphism $P\rightarrow X$ such that all fibers
are isomorphic to complex tori.
In this case each fiber is an abelian variety.
However by a \emph{family of abelian varieties} we mean the data of a family of complex
tori $E\rightarrow X$ with a section, called the \emph{zero section}.
\end{Def}

When $E\rightarrow X$ is a family of abelian varieties,
then there are global maps $E\times E\rightarrow E$ (addition)
and $E\rightarrow E$ (inverse) over $X$.
Many definitions and results from abelian varieties carry over directly to families.
A \emph{morphism} of families of abelian varieties must preserve the zero section.
An \emph{isogeny} is a sujrective morphism with finite fibers. A \emph{polarization}
is a global holomorphic $2$-form on $E$ that induces a polarization on each fiber.
A \emph{factor} of $E\rightarrow X$ is a sub-family $F$ such that there is another sub-family $G$
with addition $F\times G\rightarrow E$ being an isomorphism.

\begin{Def}
A family of abelian varieties $E\rightarrow X$ is called \emph{almost constant}
if it becomes constant after a finite étale base change.
\end{Def}

Attached to each family $P\rightarrow X$ of complex tori,
there is a family $\Alb(P/X)$ of polarized abelian varieties over $X$ called the
\emph{relative Albanese variety}, which on each fiber corresponds to the Albanese variety,
and there is a morphism $P\rightarrow \Alb(P/X)$ over $X$.

\begin{The}
Suppose that $X$ satisfies \PropP{} and let $P\rightarrow X$
be a family of complex tori.
Assume that, up to isogeny, $\Alb(P/X)$ has no non-trivial almost constant factor.
Then $P$ has property \PropP.
\end{The}

\begin{proof}
First, exactly as in the case of principal $\CC^*$-bundles,
a finite étale Galois cover of $P$ is a bundle $Q$
over some finite étale Galois cover $Y$ of $X$, whose general fiber is a finite cover of a complex torus, that is
$Q$ is a family of complex tori.
\[ \xymatrix{
Q \ar[r]^\tau \ar@{-->}[d]_q & P \ar[d]^{p} \\
Y \ar[r]_{\pi} & X } \]
We want to apply lemma~\ref{LemCriteriaAlbanese}. Let $A$ be an abelian variety
and $f:Q\rightarrow A$. Let $E:=\Alb(Q/Y)$.
Then $f$ factors through $E$ and induces a morphism of families of abelian varieties $g:E\rightarrow A\times Y$
and an injective morphism $E/\Ker(g)\hookrightarrow A\times Y$.
By rigidity (see \cite[Proposition~16.3]{MilneAV}) this implies that
$E/\Ker(g)$ is a constant family.
But with the polarization on $E$ and by Poincaré's reductibility theorem for families,
it is possible fo find a family $F$ over $Y$ such that addition $\Ker(g)\times F\rightarrow E$
is an isogeny, and $E/\Ker(g)$ is isogenous to $F$.

By functoriality of the relative Albanese construction, $F$
will project to a constant factor of $\Alb(\pi^*P / Y)$
and this corresponds to an almost constant factor of $\Alb(P/X)$.
But by our hypothesis there is no non-trivial such factor, thus $F$ is trivial
which means $g$, and therefore $f$, is constant
and $b_1(\overline{Q})=0$.
\end{proof}

\subsubsection{Symmetric spaces}
\label{SecSymmetricSpaces}
Our motivation is now the rigidity theorems
for hermitian locally symmetric spaces and for lattices in Lie groups,
which translate into the vanishing of the first Betti number.

\begin{The}
Let $\Omega=G/K$ be an irreducible hermitian symmetric space of noncompact type, where $G$ is a simple Lie group
of rank greater than $2$, $K$ is a maximal compact subgroup, and let $\Gamma\subset G$ be a torsion-free lattice.
Then $X:=\Gamma\setminus\Omega$ has property (P).
\end{The}

\begin{proof}
First it is known that $\Omega$ is simply connected
and that $X$ is a smooth quasi-projective variety,
see the Baily-Borel compactification (for example \cite{BorelJi}).
A finite étale Galois cover $Y$ of $X$ is a quotient $\Gamma'\setminus\Omega$ where $\Gamma'\subset\Gamma$ has finite index,
so $\Gamma'$ is still a torsion-free lattice in $G$ and is the fundamental group of $Y$.
Under our hypothesis it is known
by the results of Kazhdan (see \cite[p.~12]{BekkaDeLaHarpeValette}) that $G$ has property (T),
and so do $\Gamma$, $\Gamma'$. This implies that $b_1(\Gamma')=0$ and this is also $b_1(Y)$.

In case $X$ (and $Y$) is compact, we are done.
Else we take any smooth compactification $\overline{Y}$
(which may not be the Baily-Borel compactification since this one is usually not smooth)
and the natural morphism $\Gamma'=\pi_1(Y)\rightarrow \pi_1(\overline{Y})$ is surjective.
This implies that $\pi_1(\overline{Y})^{\text{ab}}$ is finite and so $b_1(\overline{Y})=0$.
\end{proof}

\nocite{PetersSteenbrink}
\nocite{ABCKT}
\bibliographystyle{alpha}
\bibliography{bibliographie}

\bigskip

\noindent Louis-Clément Lefèvre\\
Univ. Grenoble Alpes, IF, F-38000 Grenoble, France\\
CNRS, IF, F-38000 Grenoble, France\\
e-mail: \url{louisclement.lefevre@ujf-grenoble.fr}

\end{document}